\documentclass[12pt]{amsart}

\usepackage{pdfsync}         
\usepackage[active]{srcltx}  
\usepackage{tikz}
\usepackage{enumerate}
\usepackage{comment}
\usepackage{amssymb}
\usepackage{subcaption}
\usepackage{mathtools}

\usepackage[OT1]{fontenc}    %
\usepackage{graphicx}

\usepackage{palatino}
\usepackage[left=3cm,top=3.8cm,right=3cm]{geometry}        

\usepackage{xcolor}
\usepackage[pagebackref]{hyperref}
\hypersetup{
   colorlinks,
    linkcolor={red!60!black},
    citecolor={blue!60!black},
    urlcolor={blue!90!black}
}

\numberwithin{equation}{section}

\theoremstyle{plain}
\newtheorem{theorem}{Theorem}[section]
\newtheorem{corollary}[theorem]{Corollary}
\newtheorem{prop}[theorem]{Proposition}
\newtheorem{lemma}[theorem]{Lemma}

\theoremstyle{remark}

\theoremstyle{definition}
\newtheorem{definition}[theorem]{Definition}

\newcommand{\e}{\varepsilon}
\newcommand{\N}{\mathbb{N}}
\newcommand{\R}{\mathbb{R}}

\newcommand{\dist}{\mathrm{dist}}

\newcommand{\cR}{\mathcal{R}}

\newcommand{\cP}{\mathcal{P}}

\newcommand{\cN}{\mathcal{N}}
\newcommand{\cD}{\mathcal{D}}

\newcommand{\cE}{\mathcal{E}}

\newcommand{\de}{\delta}
\newcommand{\eps}{\varepsilon}

\DeclareMathOperator{\hdim}{dim_H}

\DeclareMathOperator{\lbdim}{\underline{\dim}_B}

\DeclareMathOperator{\supp}{supp}

\DeclareMathOperator{\bad}{\mathbf{Bad}}

\newcommand{\wh}{\widehat}
\newcommand{\wt}{\widetilde}

\newcommand{\M}{\mathbf{M}}
\newcommand{\T}{\mathbf{T}}

\newcommand{\Arrow}[1]{%
\parbox{#1}{\tikz{\draw[->](0,0)--(#1,0);}}
}

\newcommand{\ssto}{\Arrow{.2cm}}

\title[Pinned distance sets]{Improved bounds for the dimensions of planar distance sets}

\author{Pablo Shmerkin}

\address{Department of Mathematics and Statistics, Torcuato Di Tella University, and CONICET, Buenos Aires, Argentina}
\email{pshmerkin@utdt.edu}
\urladdr{http://www.utdt.edu/profesores/pshmerkin}
\thanks{P.S. was partially supported by Projects CONICET-PIP 11220150100355 and PICT 2015-3675 (ANPCyT)}


\begin{document}

\begin{abstract}
We obtain new lower bounds on the Hausdorff dimension of distance sets and pinned distance sets of planar Borel sets of dimension slightly larger than $1$, improving recent estimates of Keleti and Shmerkin, and of Liu in this regime. In particular, we prove that if $\hdim(A)>1$, then the set of distances spanned by points of $A$ has Hausdorff dimension at least $40/57 > 0.7$ and there are many $y\in A$ such that the pinned distance set $\{ |x-y|:x\in A\}$ has Hausdorff dimension at least $29/42$ and lower box-counting dimension at least $40/57$. We use the approach and many results from the earlier work of
Keleti and Shmerkin, but incorporate estimates from the recent work of Guth, Iosevich, Ou and Wang as additional input.
\end{abstract}

\maketitle

\section{Introduction}

Given $A\subset\R^d$, with $d\ge 2$, its \emph{distance set} is $\Delta(A)=\{|x-y|:x,y\in A\}$. If $y\in\R^d$ is given, we also define the \emph{pinned distance set} $\Delta_y(A)=\{ |x-y|: x\in A\}$. A major open problem in geometric measure theory, introduced by Falconer in \cite{Falconer85}, is whether $|\Delta(A)|>0$ whenever $A$ is a Borel set with $\hdim(A)>d/2$ (we denote Lebesgue measure by $|\cdot|$ and Hausdorff dimension by $\hdim$). A variant asks whether, under the same assumptions, $|\Delta_y(A)|>0$ for some $y\in A$. These problems remain open in all dimensions, but many new partial results have been achieved very recently \cite{Shmerkin17, IosevichLiu17, KeletiShmerkin18, Liu18,DGOWWZ18,DuZhang18, GIOW18, Liu18b}. We review only a small selection of relevant results in the plane.

All sets are assumed to be Borel. In \cite{KeletiShmerkin18}, Keleti and the author proved that if $A$ is a planar set with $\hdim(A)=s\in (1,3/2)$, then
\begin{equation} \label{eq:lower-bound-2/3s}
\hdim(\Delta_y A) \ge \frac{2}{3}s.
\end{equation}
outside of a set of $y$ of Hausdorff dimension $\le 1$ (in particular, for nearly all $y\in A$). In \cite{KeletiShmerkin18} we also proved that if $\hdim(A)>1$ then
\[
\hdim(\Delta A) \ge 2/3 + 1/54.
\]
We also established much better bounds for the \emph{packing} (or upper box-counting) dimension of $\Delta_y A$, as well as for the Hausdorff dimension of $\Delta_y A$ under additional structural assumptions on $A$ (such as an upper bound on its packing or box dimension).

Very recently, Guth, Iosevich, Ou and Wang proved in \cite{GIOW18} that if $A$ is a planar set with $\hdim(A)>5/4$ then there are many $y\in A$ such that $|\Delta_y(A)|>0$. This improves upon a well-known result of Wolff \cite{Wolff99} asserting that if $\hdim(A)>4/3$ then $|\Delta(A)|>0$. In another recent breakthrough, Liu \cite{Liu18} managed to replace $\Delta(A)$ by $\Delta_y(A)$ with $y\in A$ in Wolff's Theorem. Guth, Iosevich, Ou and Wang use Liu's approach as well as an idea of \cite{KeletiShmerkin18}, but also introduce several fundamental new insights.

Even more recently, Liu \cite{Liu18b}, building upon the results of \cite{GIOW18}, proved that if $s\in (1,5/4)$ and $\hdim(A)=s$ then there are many $y\in A$ such that
\begin{equation} \label{eq:lower-bound-Liu}
\hdim(\Delta_y A) \ge \frac{4}{3}s-\frac{2}{3}.
\end{equation}
This improves upon \eqref{eq:lower-bound-2/3s} for all $s>1$. However, if one only assumes $\hdim(A)>1$ (as in Falconer's original problem), both \eqref{eq:lower-bound-2/3s} and \eqref{eq:lower-bound-Liu} give $\hdim(\Delta_y(A))> 2/3$. This is perhaps a bit curious as both proofs rely on very different methods, and suggests that improving upon the $2/3$ is a natural problem. Here we prove:
\begin{theorem} \label{thm:main}
Let
\[
\phi(u) = \frac{29+19 u+6 u^2+8 u^3-8 u^4}{42-15 u+30 u^2-12 u^3}.
\]
If $A$ is a planar Borel set with $\hdim(A)=s\in (1,1.04)$ then there is $y\in A$ such that
\[
\hdim(\Delta_y A) \ge \phi(s-1) > 29/42 = 2/3+1/42 \approx 0.6904\ldots.
\]
In fact, the above holds for all $y\in\R^2$ outside of a set of Hausdorff dimension at most $1$.
\end{theorem}

This improves upon Liu's lower bound \eqref{eq:lower-bound-Liu} in the interval $(1,1.037)$ (note also that our exceptional set has Hausdorff dimension at most $1$ and so it is much smaller than $A$, while the exceptional set in Liu's approach can be as large as $A$ in terms of dimension). Theorem \ref{thm:main} also improves upon the estimate $\hdim(\Delta A) \ge 2/3+1/54$ for the full distance obtained in \cite{KeletiShmerkin18} under the assumption $\hdim(A)>1$. We are able to obtain a further improvement for the dimension of the full distance set, which also works for the lower box counting dimension of the pinned distance set. Even though for \emph{upper} box dimension or even packing dimension much better estimates are proved in \cite{KeletiShmerkin18}, a lower box dimension bound provides new information as it says the pinned distance sets is large at  \emph{all} small scales, as opposed to only infinitely many small scales.
\begin{theorem} \label{thm:full}
Let
\[
\chi(u)=\frac{8 (1+u) \left(5-3 u+4 u^2\right)}{57-30 u+48 u^2}.
\]
If $A$ is a planar Borel set with $\hdim(A)=s\in (1,1.06)$ then
\[
\hdim(\Delta A) \ge \chi(s-1) > 40/57 = 2/3+2/57 \approx 0.7017\ldots.
\]
Moreover, if $\mathcal{H}^s(A)>0$ then there is $y\in A$ such that
\[
\lbdim(\Delta_y A) \ge \chi(s-1).
\]
\end{theorem}

The lower bound $\chi(s-1)$ is better than that given by \eqref{eq:lower-bound-Liu} for $s\in (1,1.05]$.

We make some brief comments on the proofs. We follow the scheme of \cite{KeletiShmerkin18}. Recall the bound \eqref{eq:lower-bound-2/3s}.
Even though this bound is never better than \eqref{eq:lower-bound-Liu} for general sets, the methods of \cite{KeletiShmerkin18} provide better lower bounds in many cases, depending on the ``branching structure'' of $A$. Let $s$ be slightly larger than $1$. If there are arbitrarily small scales $r$ such that at scale $r$ the set $A$ is a union of $\approx r^{-s}$ squares of side length $r$ which are $\approx r^{-s/2}$-separated, then the methods of \cite{KeletiShmerkin18} do not give anything better than $2s/3$. However, for such well-separated sets, it is not too hard to see that the results of \cite{GIOW18} give a lower bound $\hdim(\Delta_y(A))>4/5$ for many $y\in A$ (this is related to the exponent $4/5$ for discretized well separated sets in \cite[Corollary 1.5]{GIOW18}). In \cite{KeletiShmerkin18} we proved a structural result saying, roughly, that if $\hdim(\Delta_y(A))\le 2s/3+\eta$ then $A$ resembles one of these well-separated sets in a rather technical but quantitative fashion.

The main idea of this paper is to show that the results of \cite{GIOW18} still give a lower bound $4/5-h(\eta)$ for the dimension of the (pinned) distance set of $A$ under the structural information on $A$ derived from $\hdim(\Delta_y(A))\le 2s/3+\eta$.  A similar idea was already used in the proof of the bound $\hdim(\Delta(A))\ge 2/3+1/54$ in \cite{KeletiShmerkin18}; there the estimates of Wolff \cite{Wolff99} were used as additional input. The additional gains in this paper follow from using more powerful quantitative estimates that we extract from \cite{GIOW18}.

Interestingly, both \cite{KeletiShmerkin18} and \cite{GIOW18} rely crucially on a spherical projection theorem of Orponen, \cite[Theorem 1.11]{Orponen18}. Hence the proofs of Theorems \ref{thm:main} and \ref{thm:full} use this theorem \emph{twice} in rather different (although related) forms. It is because of the use of this projection theorem that we need to assume that $\hdim(A)>1$ and we get no results when $\hdim(A)=1$.

We set up notational conventions in Section \ref{sec:notation}. In Sections \ref{sec:KS} and \ref{sec:Lipschitz} we review various results from \cite{KeletiShmerkin18}, while in Section \ref{sec:GIOW} we recall several estimates from \cite{GIOW18} and deduce some useful consequences. Finally, we complete the proofs of Theorems \ref{thm:main} and \ref{thm:full} in Section \ref{sec:proofs}.

\section{Notation}
\label{sec:notation}

We use Landau's $O(\cdot)$ notation: given $X>0$, $O(X)$ denotes a positive quantity bounded above by $C X$ for some constant $C>0$. If $C$ is allowed to depend on some other parameters, these are denoted by subscripts. We sometimes write $X\lesssim Y$ in place of $X=O(Y)$ and likewise with subscripts. We write $X\gtrsim Y$,  $X\approx Y$ to denote $Y\lesssim X$,  $X\lesssim Y\lesssim X$ respectively.

Throughout the rest of the paper, and according to the setup of \cite{KeletiShmerkin18}, we work with three parameters that we assume fixed: a large integer $T$ and small positive numbers $\e,\tau$.  The parameter $T$ indicates the scale we work with: we will decompose sets and measures in the base $2^T$. In particular, we will work with sets and measures that have a regular tree (or Cantor) structure when represented in this base: see Definition \ref{def:regular}. The parameter $\tau$ arises in the set of bad projections from \cite{KeletiShmerkin18} and we do not deal with it directly. Finally, $\e$ will denote a generic small parameter; it can play different roles at different places.

We will use the notation  $o_{T,\e,\tau}(1)=o_{T\to\infty,\e\to 0^+,\tau\to 0^+}(1)$ to denote any function $f(T,\e,\tau)$ such that
\[
f(T,\e,\tau)\ge 0 \quad\text{and}\quad  \lim_{T\to\infty,\e\to 0^+,\tau\to 0^+} f(T,\e,\tau)=0.
\]
If a particular instance of $o(1)$ is independent of some of the variables, we drop these variables from the notation.  Difference instances of the $o(1)$ notation may refer to different functions of $T,\e,\tau$, and they may depend on each other, so long as they can always be made arbitrarily small.

We will often work at a scale $2^{-T\ell}$;  it is useful to think that $\ell\to\infty$ while $T,\e,\tau$ remain fixed.

The family of Borel probability measures on a metric space $X$ is denoted by $\cP(X)$. If $0<\mu(A)<\infty$, then $\mu_A$ denotes the normalized restriction $\mu(A)^{-1}\mu|_A\in\cP(A)$. If $f:X\to Y$ is a Borel map, then by $f\mu$ we denote the push-forward measure, i.e. $f\mu(A)= \mu(f^{-1}A)$.

We let $\cD_j$ be the half-open $2^{-jT}$-dyadic cubes in $\R^d$ (where $d$ is understood from context), and let $\cD_j(x)$ be the only cube in $\cD_j$ containing $x\in \R^d$. Given a measure $\mu\in\cP(\R^d)$, we also let $\cD_j(\mu)$ be the cubes in $\cD_j$ with positive $\mu$-measure. Note that these families depend on $T$. Given $A\subset \R^d$, we also denote by $\cN(A,\ell)$ the number of cubes in $\cD_\ell$ that intersect $A$.

A $2^{-m}$-measure is a measure in $\cP([0,1)^d)$ such that the restriction to any $2^{-m}$-dyadic cube $Q$ is a multiple of Lebesgue measure on $Q$, i.e. a measure defined down to resolution $2^{-m}$. Likewise, a $2^{-m}$-set is a union of $2^{-m}$ dyadic cubes. If $\mu\in\cP(\R^d)$ is an arbitrary measure, then we denote
\[
\cR_\ell(\mu) = \sum_{Q\in\cD_\ell} \mu(Q) \text{Leb}_Q,
\]
that is, $\cR_\ell(\mu)$ is the $2^{-T\ell}$-measure that agrees with $\mu$ on all dyadic cubes of side length $2^{-T\ell}$. We also define the corresponding analog for sets:  given $A\subset\R^d$, $\cR_\ell(A)$ denotes the union of all cubes in $\cD_\ell$ that intersect $A$.

We will sometimes need to deal with supports in the dyadic metric, i.e. given $\mu\in\cP([0,1)^d)$ we let
\[
\supp_{\mathsf{d}}(\mu) = \{ x: \mu(\cD_j(x))>0 \text{ for all } j\in\N\}.
\]
Note that $\mu(\supp_{\mathsf{d}}(\mu))=1$ and that $\supp_{\mathsf{d}}(\mu)\subset \supp(\mu)$.

If a measure $\mu\in\cP(\R^d)$ has a density in $L^p$, then its density is sometimes also denoted by $\mu$, and in particular  $\|\mu\|_{L^p}$ stands for the $L^p$ norm of its density.

Logarithms are always to base $2$.

\section{The Keleti-Shmerkin framework}
\label{sec:KS}

In this section we recall several concepts and results from \cite{KeletiShmerkin18}.

\subsection{Regular measures and energy}

Following \cite{KeletiShmerkin18}, we will decompose a $2^{-T\ell}$-measure in terms of measures which have a uniform tree structure when represented in base $2^T$. This notion is made precise in the next definition.

\begin{definition} \label{def:regular}
Given a sequence $\sigma=(\sigma_1,\ldots,\sigma_{\ell})\in \R^\ell$, we say that $\mu\in \cP([0,1)^d)$ is \emph{$\sigma$-regular} if it is a $2^{-{T\ell}}$-measure, and for any $Q\in \cD_{j}(\mu)$, $1\le j\le\ell$, we have
\[
\mu(Q) \le 2^{-T(\sigma_j+1)} \mu(\wh{Q}) \le 2\mu(Q),
\]
where $\wh{Q}$ is the only cube in $\cD_{j-1}$ containing $Q$.
\end{definition}
The exponent $T(\sigma_j+1)$ is a convenient normalization. The key point in this definition is that a measure is $\sigma$-regular if all cubes of positive mass have roughly the same mass, and the sequence $(\sigma_j)$ quantifies this common mass. We have the following easy estimate for the mass decay of regular measures.
\begin{lemma} \label{lem:regular-mass-decay}
Suppose $\mu\in\cP([0,1)^2)$ is $(\sigma_1,\ldots,\sigma_\ell)$-regular and
\[
\sum_{i=1}^j (\alpha-\sigma_i) \le K \quad j=1,\ldots,\ell.
\]
Then
\[
\mu(B(x,r)) \lesssim_T 2^{K T} r^{1+\alpha} \quad\text{for all }x\in \R^2, r\in (0,1].
\]
\end{lemma}
\begin{proof}
Let $Q\in \cD_j(\mu)$. By definition of regularity and the assumption
\[
\mu(Q) \le 2^{-T(\sigma_1+\ldots+\sigma_j+j)} \le 2^{K T} 2^{-(1+\alpha)jT}.
\]
Since we can cover a ball $B(x,r)$ by $O_T(1)$ squares $Q\in\cD_j$ with $2^{-jT}\le r$, the claim follows.
\end{proof}

The decomposition we referred to above is detailed in the next proposition; see \cite[Corollary 3.5]{KeletiShmerkin18} for the proof.
\begin{prop} \label{prop:bourgain}
Fix $\ell\ge 1$, write $m=T\ell$, and let $\mu$ be a $2^{-m}$-measure on $[0,1)^2$.  There exists a family of pairwise disjoint $2^{-m}$-sets $X_1,\ldots, X_N$  with $X_i\subset\supp_{\mathsf{d}}(\mu)$, and such that:
\begin{enumerate}[(\rm i)]
\item $\mu\left(\bigcup_{i=1}^N X_i\right) \ge 1-2^{-\e m}$. In particular, if $\mu(A)> 2^{-\e m}$, then there exists $i$ such that $\mu_{X_i}(A)\ge \mu(A)-2^{-\e m}$.
\item $\mu(X_i) \ge 2^{-(\e+\log(2d T+2)/T) m} \ge 2^{-o_{T,\e}(1) m}$ for each $i$.
\item Each $\mu_{X_i}$ is $\sigma(i)$-regular for some $\sigma(i)\in  [-1,1]^\ell$.
\end{enumerate}
\end{prop}

Recall that the $s$-energy of $\mu\in\cP(\R^d)$ is
\[
\cE_s(\mu) = \iint \frac{d\mu(x)d\mu(y)}{|x-y|^s}.
\]
Energies play an important r\^{o}le in all known approaches to the Falconer distance set problem, and also in this paper. The following bound for the $s$-energy of regular measures is proved in \cite[Lemma 3.3]{KeletiShmerkin18}.
\begin{lemma} \label{lem:energy-regular}
If $\nu\in\cP([0,1)^d)$ is $\sigma$-regular for some $\sigma\in\R^\ell$ and $s\in (0,d)$, then
\[
\left|\log \cE_s(\nu)-\left( T \max_{j=1}^{\ell} \sum_{i=1}^j (s-1)-\sigma_j \right)\right|\le O(\ell) + O_{d,s,T}(1).
\]
\end{lemma}

\subsection{Box-counting estimates for pinned distance sets}

The goal of this section is to combine several results from \cite{KeletiShmerkin18} to obtain a box-counting estimate for sets $\Delta_y(A)$ in terms of certain combinatorial information about a Frostman measure $\mu$ supported on $A$.

We recall several definitions from \cite[Section 4]{KeletiShmerkin18}. They will not be directly used in this paper, but we include them for completeness.
\begin{definition} \label{def:integer-partition}
 Given $L\in\N$, a \emph{good partition} of $(0,L]$ is an integer sequence $0=N_0<\ldots<N_q=L$ such that $N_{j+1}-N_j\le N_j+1$. If additionally
\begin{equation}\label{e:integertau}
\tau N_j \le N_{j+1}-N_j \le N_j+1
\end{equation}
then $(N_i)$ is said to be a \emph{$\tau$-good partition}.

Given a finite sequence $(\sigma_1,\ldots,\sigma_L)\in \R^L$, let
\[
\mathcal{S}(\sigma) = -\min_{j=0}^L \sigma_1+\cdots+\sigma_j \ge 0.
\]
For any good partition $\mathcal{P}=(N_j)_{j=0}^q$ of $(0,L]$ and any $\sigma \in \R^L$, we denote
\[
\M(\sigma,\mathcal{P}) = \sum_{j=0}^{q-1} \mathcal{S}(\sigma|(N_j,N_{j+1}]),
\]
where $\sigma|I$ denotes the restriction of the sequence $\sigma$ to the interval $I$. Finally, given $\sigma\in \R^L$ and $\tau\in (0,1)$, we let
\[
\M_\tau(\sigma) = \min \{ \M(\sigma,\mathcal{P}) : \mathcal{P} \text{ is a $\tau$-good partition of } (0,L]\}.
\]
\end{definition}

We write $\Delta_y(x)=|x-y|$ for the pinned distance map.  Recall that $o_{T,\e}(1)$ denotes a function of $T$ and $\e$ which tends to $0$ as $T\to\infty,\e\to 0^+$, and that $\cN(B,\ell)$ is the box-counting number of $B$ at scale $2^{-T\ell}$.

\begin{prop} \label{prop:box-counting}
Let $\mu,\nu\in\cP([0,1)^2)$ have disjoint supports and satisfy $\cE_s(\mu), \cE_u(\nu)<\infty$ for some $s\in (0,2), u>\max(1,2-s)$. There exists a set $G\subset\supp\mu\times\supp\nu$ with $(\mu\times\nu)(G)\ge 2/3$ such that the following holds for $\ell_0$ sufficiently large in terms of $\mu,\nu,T,\e$.

Suppose $\ell\ge\ell_0$. Let $\rho=\mu_{X_i}$, where $X_i$ is one of the sets given by Proposition \ref{prop:bourgain} applied to $\cR_\ell\mu$. If $y\in\supp(\nu)$ and $A\subset\supp(\cR_\ell\mu)$ satisfy
\[
A\subset  \cR_\ell\{ x: (x,y)\in G\}\quad\text{and}\quad \rho(A)\ge\ell^{-2}/2,
\]
then
\[
\frac{\log \cN(\Delta_y A,\ell)}{T\ell} \ge 1 - \frac{\M_\tau(\sigma)}{\ell} -  o_{T,\e,\tau}(1).
\]
\end{prop}
\begin{proof}
The proof uses the sets of bad projections defined in \cite[\S 3.2]{KeletiShmerkin18}. We do not repeat the definitions here, but recall from \cite[Lemma 3.10]{KeletiShmerkin18} that
\[
|\bad''_{\ell_0}(\mu,x)| \lesssim_{T,\e,\tau} 2^{-\e' \ell_0}
\]
for all $x\in\supp_{\mathsf{d}}(\mu)$, where $\e'=\e'(T,\e,\tau)>0$.

Now let $\kappa=\kappa(\mu,\nu)>0$ be the number given in \cite[Proposition 3.12]{KeletiShmerkin18} (this is where the assumptions on $\mu$ and $\nu$ get used; Orponen's spherical projection theorem plays a crucial r\^{o}le here). By taking $\ell_0$ large enough we may assume that
\[
|\bad''_{\ell_0}(\mu,x)| \le \kappa
\]
for all $x\in\supp_{\mathsf{d}}(\mu)$. It follows from \cite[Proposition 3.12]{KeletiShmerkin18} that $(\mu\times\nu)(G)\ge 2/3$, where
\[
G= \{ (x,y): P_y(x) \not\in \bad''_{\ell_0}(\mu,x) \}.
\]
Now let $y$ and $A$ be as in the statement. Then for all $x\in A$ there is $\wt{x}\in \cD_\ell(x)$ such that $(\wt{x},y)\in G$ and therefore
\begin{equation} \label{eq:good-projection}
P_y(\wt{x}) \notin \bad''_{\ell_0}(\mu,\wt{x}).
\end{equation}
According to the definition of the sets $\bad'_{\ell_0\ssto\ell}(\cR_\ell\mu,x)$ and $\bad''_{\ell_0}(\mu,x)$ in \cite[Eqs. (3.3) and (3.4)]{KeletiShmerkin18}, we have $P_y(\wt{x})\notin \bad'_{\e\ell\ssto\ell}(\cR_\ell\mu,\wt{x})=\bad_{\e\ell\ssto\ell}(\rho,\wt{x})$.

The hypotheses of \cite[Proposition 4.4]{KeletiShmerkin18} are met by $\rho$ and $A$, with $\beta=\eps$. If $\ell_0$ is taken large enough in terms of $T,\e,\tau$ we can make the error term in the proposition equal to $o_{T,\e}(1)$. The claim follows from an application of \cite[Proposition 4.4]{KeletiShmerkin18}.
\end{proof}

\section{Combinatorics of $1$-Lipschitz functions}
\label{sec:Lipschitz}

We next recall some results from \cite{KeletiShmerkin18} that will help us deal with the numbers $\M_\tau(\sigma)$ from the last section. It turns out to be convenient to work with $1$-Lipschitz functions instead of $[-1,1]$-sequences; see Lemma \ref{l:sigmatof} at the end of this section for the connection between functions and sequences. The following provides an analog to Definition \ref{def:integer-partition} for Lipschitz functions (we will not directly use this definition).
\begin{definition}\label{d:many}
A sequence $(a_n)_{n=0}^\infty$ is a \emph{partition} of
the interval $[0,a]$ if $a=a_0>a_1>\ldots>0$ and $a_n\to 0$;
it is a \emph{good partition} if
we also have $a_{k-1} / a_{k} \le 2$ for every $k\ge 1$.

Let $f:[0,a]\to\R$ be continuous and $(a_n)$ be a  partition of $[0,a]$. By the
  \emph{total drop of} $f$ \emph{according to} $(a_n)$ we mean
\[
\T(f,(a_n))=\sum_{n=1}^\infty f(a_n)-\min_{[a_n,a_{n-1}]} f,
\]
and we also introduce the notation
\[
\T(f)=\inf\{\ \T(f,(a_n))\ :\ (a_n) \textrm{ is a good partition of } [0,a]\ \},
\]
\end{definition}

Although we will not use it directly, we recall \cite[Proposition 5.2]{KeletiShmerkin18} (or rather the special case in which $a=1$ and $C=1$).
\begin{prop}\label{p:basic}
Let
$u\in [0,1/2]$ be a parameter.
  Let $f:[0,1]\to\R$ be a $1$-Lipschitz function such that
  $f(x)\ge ux$ for every $x\in[0,1]$.
Then
\[
    \T(f)\le \frac{1-2u}{3}.
\]
\end{prop}
This proposition is sharp: if
\[
f(x) =  \left\{ \begin{array}{ccc}
          x & \text{ if } & x\in [0,(u+1)/2] \\
          1+u-x & \text{ if } & x\in [(u+1)/2,1]
        \end{array} \right.,
\]
then $\T(f)=(1-2u)/3$. This estimate leads to the bound $\hdim(\Delta_y(A))\ge 2s/3$ for some $y\in A$ if $\hdim(A)=s>1$. In order to improve upon this, we need another result from \cite{KeletiShmerkin18} that asserts that if $\T(f)$ is close to $(1-2u)/3$ then $f$ is close to the above example in a quantitative way.

\begin{prop}\label{p:stability}
Fix $u\in[0,1/3]$, $\eta\in(0,1/21]$ and let
$f:[0,1]\to\R$ be a $1$-Lipschitz function
such that $f(0)=0$, $f(x)\ge u x$ on $[0,1]$ and
$\T(f) > \frac{1-2u}{3}-\eta$.
Let
\begin{equation} \label{eq:t_1}
t_1= \frac{1+u}{3} - \eta\left(\frac{1+u}{1-2u}-(1-u)\right).
\end{equation}
Then
\[
f(x) \ge \left\{
\begin{array}{ll}
  ux & \text{ on } [0,3\eta] \\
  x-3\eta(1-u) & \text{ on } [3\eta,t_1] \\
 t_1-3\eta(1-u) & \text{ on } [t_1,2t_1-6\eta(1-u)]\\
  3t_1 -9\eta(1-u) -x &\text{ on } [2t_1-6\eta(1-u),2t_1-3\eta(1-u)]\\
   x-t_1-3\eta(1-u) & \text{ on } [2t_1-3\eta(1-u),2t_1]\\
  3t_1-x - 3\eta(1-u) & \text{ on } [2t_1,1]
\end{array}
\right..
\]
\end{prop}
\begin{proof}
It is proved in \cite[Proposition 5.15]{KeletiShmerkin18} that under our assumptions
\[
f(x) \ge \left\{
\begin{array}{ll}
  x-3\eta(1-u) & \text{ on } [0,t_1] \\
 t_1-3\eta(1-u) & \text{ on } [t_1,2t_1-6\eta(1-u)]\\
  3t_1-x - 3\eta(1-u) & \text{ on } [2t_1,1]
\end{array}
\right..
\]
Using the assumption $f(x)\ge ux$, Proposition \ref{p:stability} and the $1$-Lipschitz property of $f$, we get the estimates on the remaining intervals.
\end{proof}

The following corollary will be used in the proof of Theorem \ref{thm:main}. Its deduction from Proposition \ref{p:stability} is very similar to the proof of \cite[Corollary 5.17]{KeletiShmerkin18}.

\begin{corollary} \label{c:1/42}
Fix $u\in [0,1/25]$. Let $f:[0,1]\to\R$ be a $1$-Lipschitz function such that $f(0)=0$ and $f(x)\ge u x$ for all $x\in [0,1]$. Let
\begin{align*}
\eta &= \frac{1+u-4 u^2-4 u^3}{42-15 u+33 u^2-21 u^3+6 u^4} = \phi(u) - \frac{2(1+u)}{3},\\
\xi &= \frac{3 \left(1-u-2 u^2-   \eta(2+7u-5u^2+2u^3)\right)}{(4-u) (1-2 u)} \in (2/3,1).
\end{align*}
Then either $\T(f)\le (1-2u)/3-\eta$ or
\[
f(x) \ge \frac{1-u}{3} x - (1-2u) \eta \quad\text{on } [0,\xi].
\]
\end{corollary}
\begin{proof}
Assume $\T(f)\le (1-2u)/3-\eta$. By Proposition \ref{p:stability}, and comparing slopes, it is enough to prove the desired inequality at $x=3\eta, x=2t_1-3\eta(1-u)$ and $x=\xi$, where $t_1$ is given by \eqref{eq:t_1}. Recalling that $u\in [0,1/25]$ we can check that $\eta<1/40$, which in turn gives $t_1\ge 1/3$. We can also see that $\xi>2t_1$. Using Proposition \ref{p:stability} we can then verify that, indeed,
\begin{align*}
f(3\eta) &\ge 3\eta u = \frac{1-u}{3} (3\eta) - (1-2u)\eta,\\
f(2t_1-3\eta(1-u)) & \ge t_1-6\eta(1-u) > \frac{1-u}{3}(2t_1-3\eta(1-u)) - (1-2u)\eta\\
f(\xi) &\ge 3t_1-\xi-3\eta(1-u) = \frac{1-u}{3} \xi - (1-2u)\eta.
\end{align*}
(In fact $\xi$ was defined so that the last equality is satisfied.)
\end{proof}

The following variant, which will be used to prove Theorem \ref{thm:full}, has a nearly identical proof.
\begin{corollary} \label{c:2/57}
Fix $u\in [0,0.06]$. Let $f:[0,1]\to\R$ be a $1$-Lipschitz function such that $f(0)=0$ and $f(x)\ge u x$ for all $x\in [0,1]$. Let
\begin{align*}
\eta &= \frac{2 (1+u) (1-2 u)}{57-30 u+48 u^2} =  \chi(u)-\frac{2(1+u)}{3}\\
\xi &= \frac{4-4 u-8 u^2-  \eta(9+30u-24u^2)}{5 (1-2 u)}
\end{align*}
Then either $\T(f)\le (1-2u)/3-\eta$ or
\[
f(x) \ge \frac{x-3(1-4u)\eta}{4} \quad\text{on } [0,\xi].
\]
\end{corollary}
\begin{proof}
Assume $\T(f)\le (1-2u)/3-\eta$. Applying Proposition \ref{p:stability} as in the previous corollary it is enough to check the claimed inequality for $x=3\eta, x=2t_1-3\eta(1-u)$ and $x=\xi$. Using that $u\in [0,0.06]$ we get $\eta<1/29$ and $t_1>1/3$. Using this and Proposition \ref{p:stability},
\begin{align*}
f(3\eta) &\ge 3\eta u = \frac{3\eta-3(1-4u)\eta}{4}\\
f(2t_1-3\eta(1-u)) & \ge t_1-6\eta(1-u) > \frac{2t_1-3\eta(1-u)-3(1-4u)\eta}{4}\\
f(\xi) &\ge 3t_1-\xi-3\eta(1-u) = \frac{\xi-3(1-4u)\eta}{4}\eta,
\end{align*}
where again $\xi$ was defined precisely so that the last equality is satisfied.
\end{proof}

We conclude this section with a lemma that will help us translate the results for Lipschitz functions to results for $[-1,1]$-sequences; it is the special case of \cite[Lemma 5.22]{KeletiShmerkin18} in which $L=\ell$.
\begin{lemma}\label{l:sigmatof}
Let $\gamma,\Gamma\in[-1,1]$, $\tau\in(0,1/2)$,
$\zeta\in(0,1)$ and
let $\sigma\in [-1,1]^\ell$ satisfy
\[
\gamma j - \zeta\ell \le
\sigma_1 +\ldots + \sigma_j \le
\Gamma j + \zeta\ell  \quad (1\le j\le \ell).
\]

Then there exists a piecewise linear $1$-Lipschitz
function $f:[0,1]\to\R$ such that
\begin{enumerate}[(a)]
\item \label{i:fdef}
$f(j/\ell)=\frac{1}{\ell}(\sigma_1+\ldots+\sigma_j)$
if $\sqrt\zeta \ell \le j \le \ell$,
\item \label{i:estimates}
    $(\gamma-\sqrt\zeta) x \le f(x) \le
       (\Gamma +\sqrt\zeta) x$ on $[0,1]$ and
\item \label{i:MT}
\[
\frac{1}{\ell} \M_\tau(\sigma) \le
\T(f)
+2\sqrt{\zeta}+ 144\tau + O_{\tau}(\log\ell /\ell).
\]
\end{enumerate}
\end{lemma}

\section{The Guth-Iosevich-Ou-Wang estimates and some consequences}
\label{sec:GIOW}

Let $\mu,\nu\in\cP([0,1)^2)$. We are interested in estimating the dimension of $\Delta_y(\supp(\mu))$ for $\nu$-typical $y$ or, rather, discretized versions of this, under suitable assumptions on $\mu$ and $\nu$. A key innovation of Guth, Iosevich, Ou and Wang \cite{GIOW18} is a decomposition $\mu=\mu_{\text{good}}+\mu_{\text{bad}}$ where $\Delta_y\mu_{\text{bad}}$ has small mass, and $\Delta_y\mu_{\text{good}}$ has small $L^2$ norm, in both cases for $y$ in a set of large $\nu$-measure. For simplicity, we will denote $\wt{\mu}=\mu_{\text{good}}$ (note also that what we denote by $\mu$ and $\nu$ are denoted $\mu_1, \mu_2$ in \cite{GIOW18}).

When $\mu$ is a general probability measure, $\wt{\mu}$ is a complex-valued distribution. We will apply the results of \cite{GIOW18} in the case in which $\mu$ is a $2^{-T\ell}$-measure, and in this case $\wt{\mu}$ can be seen as a (complex-valued) absolutely continuous measure. The following quantitative estimates are implicitly proved in the course of the proofs of \cite[Proposition 2.1 and Proposition 2.2]{GIOW18}.

\begin{theorem}[\cite{GIOW18}] \label{thm:GIOW}
Suppose $\mu,\nu\in\cP([0,1)^2)$ satisfy
\begin{align*}
\cE_{s_0}(\mu) &\le K_{\mu,s_0}\\
\cE_{s_0}(\nu) &\le K_{\nu,s_0}\\
\nu(B(x,r)) &\le \wt{K}_{\nu,s_1} r^{s_1} \,\text{for all }x\in\R, r>0
\end{align*}
for some $s_0,s_1>1$. Assume also that $\mu$ has a bounded density and
\[
\dist(\supp(\mu),\supp(\nu)) \gtrsim 1.
\]
Then for all large $R$ and small $\delta>0$ there is a function $\wt{\mu}:\R^2\to\mathbb{C}$ such that the following holds:
\begin{enumerate}
  \item There is a set $B$ with $\nu(B)\ge 1- R^{\delta/O(1)}$ such that
  \[
  \int |\Delta_y\mu(t)-\Delta_y\wt{\mu}(t)| \,dt \lesssim_{\delta} (K_{\mu,s_0}K_{\nu,s_0})^{1/2} R^{-\delta/O(1)}
  \]
  for each $y\in B$.
  \item
  \[
  \int \| \Delta_y\wt{\mu}\|_2^2 \,d\nu(y) \lesssim_{\delta} \wt{K}_{\nu,s_1}^{1/3} R^{O(1)} \,\mathcal{E}_{(5-s_1+O(\delta))/3}(\mu).
  \]
\end{enumerate}
The implicit constants may depend on the distance between $\supp(\mu)$ and $\supp(\nu)$ as well as $s_0$ and $s_1$.
\end{theorem}
\begin{proof}
We indicate how to deduce these estimates by following the proofs of \cite[Propositions 2.1 and 2.2]{GIOW18}. In that paper $s_0=s_1$ and the Frostman constants of $\mu$ and $\nu$ (essentially our $K$'s) remain fixed so the authors do not pay explicit attention to how other quantities depend on them; however, it is not hard to track the exact dependencies. Likewise, $R_0$ (our $R$) and $\delta$ in \cite{GIOW18} are respectively large and small but ultimately fixed, while we need explicit estimates in terms of these parameters, which again are not hard to extract from the proofs of \cite{GIOW18}.

We start by noting that in the proofs of \cite[Propositions 2.1 and 2.2]{GIOW18} the values of $R$ and $\delta$ and therefore the measure $\wt{\mu}$ are the same, but otherwise the proofs are independent from each other. In particular, we are allowed to consider different exponents $s_0,s_1$ in each proof. Also, while \cite[Proposition 2.2]{GIOW18} has an assumption $s_1>5/4$ (in our notation), this is not used in the proof until the very end; for our estimates to be valid in fact $s_1>0$ is enough (on the other hand, $s_0>1$ is key as we will see shortly).

In the proof of \cite[Lemma 3.6]{GIOW18}, the last line of the proof shows that the implicit constant in the statement is given by
\[
O(1) \int \| P_y \mu\|_{L^p} \, d\nu(y),
\]
where $P_y$ denotes spherical projection with center $y$ and $p=p(s_0)>1$. The fact that this integral is finite under our assumptions was proved by Orponen in \cite[Theorem 1.11]{Orponen18}. In fact, Orponen provides a quantitative estimate in \cite[Eq. (3.6)]{Orponen18}:
\[
 \int \| P_y \mu\|_{L^p}^p \, d\nu(y) \lesssim \cE_{s_0}(\nu)^{1/(2p)} \cE_{s_0}(\mu)^{1/2}.
\]
Hence the implicit constant in \cite[Lemma 3.6]{GIOW18} is $O_{s_0}(1) (K_{\mu,s_0} K_{\nu,s_0})^{1/2}$. The first claim follows by inspecting the proof of \cite[Proposition 2.1]{GIOW18} assuming Lemma 3.6.

We turn to the second part. Looking at \cite[Eq. (5.1) and following display]{GIOW18}, we see that the claim will follow by establishing that $\e=O(\delta)$, the constant $C(R)=R^{O(1)}$, and the implicit constant in \cite[Proposition 5.3]{GIOW18} is $O(K_{\nu,s_1}^{1/3})$ .  The first two facts can be easily read off the proof of \cite[Proposition 5.3]{GIOW18}; in particular see the last line of the proof for the value of $C(R)$.

In the proof of \cite[Proposition 5.3]{GIOW18} , the Frostman condition on $\nu$ is used at a single point: to estimate $\| \nu*\eta_{1/r}\|_{L^\infty}\lesssim r^{2-s_1}$ where $r$ is large and $\eta_{1/r}$ is a mollifier (recall that $\mu_2$ and $\alpha$ in \cite{GIOW18} are our $\nu$ and $s_1$). The explicit bound is $\| \nu*\eta_{1/r}\|_{L^\infty}\lesssim \wt{K}_{\nu,s_1}\, r^{2-s_1}$. Following through with the proof, we see that the factor that appears in the final estimate is actually $\| \nu*\eta_{1/r}\|_{L^\infty}^{1/3}$, and so one ultimately gets an additional factor $\wt{K}_{\nu,s_1}^{1/3}$.

\end{proof}

The following corollary of Theorem \ref{thm:GIOW} is deduced in essentially the same way as in the proof of \cite[Theorem 1.2]{GIOW18} from \cite[Propositions 3.1 and 3.2]{GIOW18}.

\begin{corollary} \label{cor:GIOW}
Under the same assumptions of Theorem \ref{thm:GIOW}, the following holds: if $C_1=C_1(\delta)$ and $C_2$ are sufficiently large, then there exists a set $B$ with
\[
\nu(B) > 1 - R^{-\delta/C_2}
\]
such that if $y\in B$ and $A$ is any set satisfying
\[
\mu(A) > C_1 (K_{\mu,s_0} K_{\nu,s_0})^{1/2} R^{-\delta/C_2},
\]
then
\[
|\Delta_y(A)|^{-1} \le \mu(A)^{-2} \,\wt{K}_{\nu,s_1}^{1/3} \, R^{C_2} \,\cE_{(5-s_1+C_2\delta)/3}(\mu).
\]
\end{corollary}
\begin{proof}
By the first part of Theorem \ref{thm:GIOW}, $\nu(B')\ge 1- R^{-\delta/C_2}$, where
\[
B'=\left\{ y: \int | \Delta_y\mu(t) - \Delta_y\wt{\mu}(t) | dt \le \frac{1}{2} C_1 (K_{\mu,s_0} K_{\mu,s_1})^{1/2} R^{-\delta/C_2}\right\},
\]
and $C_1$ and $C_2$ as in the statement. On the other hand, for each $y\in B'$ we have
\begin{align*}
\mu(A) &\le \int_{\Delta_y(A)} \Delta_y\mu(t) \,dt \\
&\le \int_{\Delta_y(A)} |\Delta_y\mu(t)-\Delta_y\wt{\mu}(t)|\,dt+ \int_{\Delta_y(A)} |\Delta_y\wt{\mu}(t)|\,dt\\
&\le \frac{1}{2}  \mu(A) + \int_{\Delta_y(A)} |\Delta_y\wt{\mu}(t)|,
\end{align*}
using the assumption on $\mu(A)$ in the last line. Using Cauchy-Schwarz, we deduce
\[
\mu(A)^2/4 \le \left(\int_{\Delta_y(A)} |\Delta_y\wt{\mu}(t)| dt\right)^2 \le |\Delta_y(A)| \|\Delta_y\wt{\mu}\|_{L^2}^2
\]
for all $y\in B'$. By the second part of Theorem \ref{thm:GIOW} and Markov's inequality, $\nu(B'')>1-R^{-1}$, where
\[
B'' = \left\{y: \|\Delta_y\wt{\mu}\|_{L^2}^2
\le  \wt{K}_{\nu,s_1}^{1/3} R^{O(1)} \cE_{(5-s_0+O(\delta))/3}(\mu) \right\}.
\]
The claim follows by taking $B=B'\cap B''$.
\end{proof}

The following result is a further corollary of Theorem \ref{thm:GIOW} that is well suited to the proof of Theorem \ref{thm:main}.
\begin{prop} \label{prop:GIOW-multiscale}
Let $s,s_1>1$. Suppose $\nu\in\cP([0,1)^2)$ satisfies
\[
\nu(B(x,r)) \le C\, r^{s_1}  \quad\text{for all }x\in\R,r>0.
\]
Let also $\mu\in\cP([0,1)^2)$ and assume that for each $x\in\supp_{\mathsf{d}}(\mu)$ and for each sufficiently large $\ell$ we are given a $2^{-T\ell}$-measure $\rho_{x,\ell}$ such that
\[
\cE_s(\rho_{x,\ell}) \le  2^{o_{T,\e}(1)T\ell}
\]
and
\[
\dist(\supp(\nu),\supp(\rho_{x,\ell})) \gtrsim 1.
\]

If $\ell_0$ is large enough (in terms of $T,\e$), then there exists a set $G\subset \supp(\mu)\times\supp(\nu)$ with $(\mu\times\nu)(G)\ge 9/10$ such that for $(x,y)\in G$ the following holds: if $\ell\ge \ell_0$ and $\rho_{x,\ell}(F)>\ell^{-2}/5$ for some $2^{-T\ell}$-set $F$, then
\[
|\Delta_y(F)| > 2^{-o_{T,\e}(1)T\ell} / \cE_{(5-s)/3}(\rho_{x,\ell}).
\]
\end{prop}
\begin{proof}
Fix $x$ and $\ell$ for the time being and write $\rho=\rho_{x,\ell}$ for simplicity. We are in the setting of Corollary \ref{cor:GIOW} applied to $\rho$ (in place of $\mu$) and $\nu$, with $s_0\in (1,\min(s,s_1))$, $K_{\rho,s_0}=O(1) 2^{o_{T,\e}(1)T\ell}$, and $K_{\nu,s_0}=\wt{K}_{\nu,s}=O(1)$. Here we are using that a Frostman condition with exponent $s_1$ implies (quantitative) finiteness of the energy for $s_0<s_1$.

We can then fix $\delta = o_{T,\e}(1)$ and $R=2^{o_{T,\e}(1)T\ell}$ in such a way that
\[
\ell^{-2}/5 > C_1 (K_{\rho,s_0}K_{\nu,s_0})^{1/2} R^{-\delta/C_2}\quad\text{and}\quad R^{-\delta/C_2}\le 2^{-T\ell/O_{T,\e}(1)}.
\]
Let $B_{x,\ell}=\cR_\ell B$, where $B$ is the set given by Corollary \ref{cor:GIOW} applied with these parameters. Then
\[
\nu(B_{x,\ell})\ge 1- 2^{-T\ell/O_{T,\e}(1)}.
\]
Moreover, if $y\in B_{x,\ell}$ and $F$ is a $2^{-T\ell}$-set with $\rho_{x,\ell}(F)\ge \ell^{-2}/2$, then
\begin{align*}
|\Delta_y(F)| &\ge 2^{-o_{T,\e}(1)T\ell}/\cE_{(5-s+o_{T,\e}(1))/3}(\rho_{x,\ell}) \\
&\gtrsim  2^{-o_{T,\e}(1)T\ell}/\cE_{(5-s)/3}(\rho_{x,\ell}),
\end{align*}
using Lemma \ref{lem:energy-regular} and the fact that $\rho_{x,\ell}$ is a $2^{-T\ell}$-measure. Here we are using that for a $2^{-m}$-set $F$, if $|y-y'|\le 2^{-m}$ then $|\Delta_y(F)| \le O(1) |\Delta_{y'}(F)|$.

Setting $B_{\ell_0}(x)=\cap_{\ell \ge \ell_0} B_{x,\ell}$, we have that $\nu(B_{\ell_0}(x)) \ge 9/10$ for all $x\in \supp_{\mathsf{d}}(\mu)$, again provided $\ell_0$ is large enough. The set $G=\{ (x,y): y\in B_{\ell_0}(x)\}$ has the desired properties.
\end{proof}

\section{Proofs of main theorems}
\label{sec:proofs}

\subsection{Proof of Theorem \ref{thm:main}}

In this section we prove Theorem \ref{thm:main}. We fix $s\in (1,1.04)$ for the rest of the proof; all implicit constants may depend on it. Since $\phi$ is continuous, it is enough to prove the desired conclusion under the assumption $\hdim(A)>s$ (rather than $\hdim(A)=s$).

Let $u=s-1$, and let $\eta$ and $\xi$ be the numbers given by Corollary \ref{c:1/42} for this value of $u$. We let $\phi=\phi(u)$ where $\phi$ is the function in the statement of the theorem. We note the identity
\begin{equation} \label{eq:relationship-phi}
\phi = \xi - (1-2u)\eta = \frac{2(1+u)}{3} + \eta.
\end{equation}
Indeed, $\eta$ and then $\phi$ were defined precisely so that this holds.

Recall that $o_{T,\e,\tau}(1)$ stands for a function of $T,\e,\tau$ which tends to $0$ as $T\to\infty$ and $\e,\tau\to 0^+$. We will henceforth assume that $T,\e,\tau$ are given, and that the integer $\ell_0$ is chosen large enough in terms of $T,\e,\tau$ that any required bounds on $\ell_0$ to hold.

It is enough to show that if $\nu,\mu\in\cP([0,1)^2)$ have disjoint supports and satisfy Frostman conditions
\begin{align*}
\nu(B(x,r)) &\lesssim r^{s_0},\\
\mu(B(x,r)) &\lesssim r^s,
\end{align*}
for some $s_0>1$, then there is $y\in\supp(\nu)$ (possibly depending on $T,\e,\tau$) such that
\[
\hdim(\Delta_y(\supp(\mu))) \ge \phi - o_{T,\e,\tau}(1).
\]
See e.g. \cite[Proof of Theorem 1.2]{KeletiShmerkin18} for details on this standard reduction. The measures $\nu,\mu$ are fixed for the rest of the proof and any implicit constants may depend on them.

Let $G'\subset \supp\mu\times\supp\nu$ be the set $G$ given by Proposition \ref{prop:box-counting}. In particular, $(\mu\times\nu)(G')>2/3$.

Given $\ell\ge\ell_0$ let $\ell' = \lfloor \xi\ell\rfloor$. We consider, for each $\ell$, the decomposition $(X_{\ell,i})$ given by Proposition \ref{prop:bourgain} applied to $\cR_\ell\mu$. We denote $\rho_{\ell,i}=(\cR_\ell\mu)_{X_{\ell,i}}$ and, for $x\in X_{\ell,i}$, we write
\[
\rho'_{x,\ell'} = \cR_{\ell'}\rho_{\ell,i}.
\]
(If $x\notin\cup_i X_{\ell,i}$, we define $\rho'_{x,\ell'}=\mu$ for completeness). Note that if $\rho_{\ell,i}$ is $\sigma$-regular, then $\cR_{\ell'}\rho_{\ell,i}$ is $(\sigma_1,\ldots,\sigma_{\ell'})$-regular. Using (ii) in Proposition \ref{prop:bourgain},
\[
\rho'_{x,\ell'}(B(y,r)) \lesssim \rho_{\ell,i}(B(y,r)) \lesssim 2^{o_{T,\e}(1)T\ell} \mu(B(y,r))\lesssim 2^{o_{T,\e}(1)T\ell} r^s.
\]
Since $\supp\rho'_{x,\ell'}$ is contained in the $O(2^{-\ell'})$-neighborhood of $\supp\mu$, we also have
\[
\dist(\supp\nu,\supp\rho'_{x,\ell'}) \gtrsim 1.
\]
We have checked the assumptions of Proposition \ref{prop:GIOW-multiscale} for $\nu$ and $\{ \rho'_{x,\ell'}\}$ (in fact, $\ell\to\ell'$ is not a bijection; to remedy this we can for example restrict our attention to a subset of $\ell$ on which it is). Let $G''\subset\supp\mu\times\supp\nu$ be the set given the proposition.

Define $G=G'\cap G''$. Then $(\mu\times\nu)(G)>1/2$. We henceforth fix a point $y$ with $\mu(A_1)>1/2$ for the rest of the proof, where $A_1=\{ x: (x,y)\in G\}\subset \supp(\mu)$. Our goal is to show that
\[
\hdim(\Delta_y A_1) \ge \phi-o_{T,\e,\tau}(1),
\]
which will clearly imply the statement.

In turn, by a standard dyadic pigeonholing argument (see e.g. \cite[Lemma 6.2]{KeletiShmerkin18} or \cite[Lemma 3.1]{Liu18b}), it is enough to show that if $A_2\subset A_1$ satisfies $\mu_{A_1}(A_2)\ge \ell^{-2}$ and $\ell$ is large enough, then
\begin{equation} \label{eq:box-counting-lower-estimate}
\log \cN(\Delta_y A_2 ,\ell) \ge (\phi-o_{T,\e,\tau}(1))T\ell.
\end{equation}

Since the set $\Delta_y(\cR_\ell A_2)$ is contained in the $(\sqrt{2}\cdot 2^{-T\ell})$-neighborhood of $\Delta_y A_2$, the numbers $\log \cN(\Delta_y A_2,\ell)$ and $\log \cN(\Delta_y  \cR_\ell A_2,\ell)$ differ by at most a constant. We can then assume that $A_2$ is a $2^{-T\ell}$-set. Moreover, replacing $A_2$ by $A_2 \cap \cR_\ell(A_1)$, we may assume that $A_2\subset \cR_\ell(A_1)$.

Consider again the sets $(X_{\ell,i})$ given by Corollary \ref{prop:bourgain} applied to $\cR_{\ell}\mu$. Using the first part of the corollary with $A=A_2$, and that
\[
2^{-\e T\ell}\ll  \tfrac{1}{2}\ell^{-2} \le  \mu(A_1)\mu_{A_1}(A_2) = \mu(A_2)=\cR_\ell\mu(A_2)
\]
for large enough $\ell$, we get that there is some $i$ such that, setting $X=X_{\ell,i}$ and $\rho=\rho_{\ell,i}=(\cR_\ell\mu)_{X}$,
\begin{enumerate}[(\rm i)]
\item $\rho(A_2) \ge \ell^{-2}/3$.
\item $\cR_\ell\mu(X) \ge 2^{-o_{T,\e}(1) T\ell}$ and therefore
\[
\cE_s(\rho) \le (\cR_\ell\mu(X)^{-2}\cE_s(\cR_\ell\mu) \lesssim_T 2^{o_{T,\e}(1) T\ell} \cE_s(\mu) \lesssim 2^{o_{T,\e}(1) T\ell}.
\]
\item $\rho$ is $\sigma$-regular for some sequence $\sigma=(\sigma_1,\ldots,\sigma_\ell)$, $\sigma_j\in [-1,1]$.
\item $X$ is contained in $\cR_\ell\supp_{\mathsf{d}}(\mu)$.
\end{enumerate}

Recall that $G\subset G'$ where $G'$ is the set from Proposition \ref{prop:box-counting}. Note that $\rho$ and $A_2$ satisfy the assumptions of the proposition, so  (assuming $\e>0$ is small enough to ensure separation) we get that
\begin{equation} \label{eq:lower-box-multiscale}
\log \cN(\Delta_y A_2,\ell) \ge \left(1 - \frac{\M_\tau(\sigma)}{\ell} -  o_{T,\e,\tau}(1)\right)T\ell.
\end{equation}
Denote $\rho'=\rho'_{x,\ell}=\cR_{\ell'}\rho$. Using the fact that the set $G\subset G''$ was defined so that Proposition \ref{prop:GIOW-multiscale} holds, we also have that
\[
\log|\Delta_y(A_2)| \ge  -\log\left( \cE_{(5-s)/3}(\rho')\right) -o_{T,\e}(1)T\ell'
\]
and therefore
\begin{equation} \label{eq:lower-box-GIOW}
\log \cN(\Delta_y A_2,\ell) \ge \log \cN(\Delta_y A_2,\ell') \ge (1-o_{T,\e}(1))T\ell' - \log\left( \cE_{(5-s)/3}(\rho')\right).
\end{equation}

From the results of \cite{KeletiShmerkin18} it follows the right-hand side of \eqref{eq:lower-box-multiscale} is at least $(\tfrac{2s}{3} -o_{T,\e}(1))T\ell$. The idea is to show that if that right-hand side is too close to $\tfrac{2s}{3} T\ell$, then one gets an improvement using \eqref{eq:lower-box-GIOW}.

By Lemma \ref{lem:energy-regular} and (ii), (iii) above,  and assuming that $\ell_0$ was taken large enough in terms of $T$, we have
\[
  \sum_{i=1}^j \sigma_i \ge (s-1)j - \ell o_{T,\e}(1) \qquad(j=1,\ldots,\ell).
\]
Lemma \ref{l:sigmatof} applies to give a $1$-Lipschitz function $f:[0,1]\to \R$ with
\begin{enumerate}[(a)]
  \item $f(j/\ell) = \tfrac{1}{\ell}(\sigma_1+\ldots+\sigma_j)$ for $j\in [o_{T,\e}(1)\ell ,\ell]$,
  \item $f(x) \ge (s-1-o_{T,\e}(1))x$,
  \item  $\frac{1}{\ell} \M_\tau(\sigma) \le \T(f)+o_{T,\e,\tau}(1)$.
\end{enumerate}
Suppose first that $\T(f)\le 1-\phi+ 2o_{T,\e}(1)/3$, where this particular instance of $o_{T,\e}(1)$ is the one from (b). Then \eqref{eq:lower-box-multiscale} and (c) immediately yield \eqref{eq:box-counting-lower-estimate}.

Hence from now on we assume that
\[
\T(f)> 1-\phi+ \frac{2 o_{T,\e}}{3}= \frac{1-2(u-o_{T,\e}(1))}{3}-\eta ,
\]
using the identity \eqref{eq:relationship-phi}. Thanks to (b) above, we can apply Corollary \ref{c:1/42} with $u-o_{T,\e}(1)$ in place of $u$. Since $\eta$ and $\xi$ are smooth functions of $u$, we deduce that
\[
f(x) \ge \frac{1+u}{3} x-(1-2u)\eta - o_{T,\e}(1) \quad\text{on } [0,\xi].
\]
Using (a) above, this implies that
\[
\sum_{i=1}^j \frac{1+u}{3}-\sigma_i \le \ell( (1-2u)\eta +o_{T,\e}(1)),\quad j=1,\ldots, \ell'.
\]
By Lemma \ref{lem:energy-regular}, this says that $\rho'$, which we recall is a $(\sigma_1,\ldots,\sigma_{\ell'})$-regular measure, satisfies
\[
\log\cE_{(4-u)/3}(\rho') \le T\ell( (1-2u)\eta+o_{T,\e}(1)).
\]
In light of \eqref{eq:lower-box-GIOW}, and recalling that $\ell'=\lfloor \xi\ell\rfloor$ and the identity \eqref{eq:relationship-phi}, we also get \eqref{eq:box-counting-lower-estimate} in this case, and this completes the proof of Theorem \ref{thm:main}.

\subsection{Proof of Theorem \ref{thm:full} (Hausdorff dimension part)}
\label{subsec:full-Hausdorff}

Now we being the proof of Theorem \ref{thm:full}. As before, we assume that $\hdim(A)>s$ and we aim to prove  that
\begin{equation} \label{eq:full-dist-Hausdorff-dim-claim}
\hdim(\Delta(A)) \ge \chi(s-1).
\end{equation}
In the next section we indicate how to get the statement about box dimension of pinned distance sets.

We henceforth fix $u=s-1$, $\chi=\chi(u)$ and let $\eta,\xi$ be the values given by Corollary \ref{c:2/57}. We note the identities
\begin{equation} \label{eq:identity-chi}
\chi = \xi - (1-4u)\eta = \frac{2(u+1)}{3}+\eta.
\end{equation}
Again, $\eta$ and then $\chi$ were defined precisely so that this holds.

As usual fix $T\gg 1, \e,\tau\ll 1$. It is enough to consider the case in which $A\subset [0,1)^2$. Let $\mu_1,\mu_2\in\cP([0,1)^2)$ be measures supported on $A$ such that $\cE_{s}(\mu_1),\cE_{s}(\mu_2)<\infty$, and their supports are at distance $\gtrsim 1$. Any implicit constants arising in the proof may depend on $\mu_1$, $\mu_2$ and $s$.

Let $G_1, G_2 \subset \supp\mu_1\times\supp\mu_2$  be the sets given by Proposition \ref{prop:box-counting} applied to $\mu=\mu_1$ and $\nu=\mu_2$ and to $\mu=\mu_2$ and $\nu=\mu_1$ respectively (in the second case we swap the coordinates to get a subset of $\supp\mu_1\times\supp\mu_2$). We set $G=G_1\cap G_2$ and note that $(\mu_1\times\mu_2)(G)>1/3$.

Write $\Delta(x,y)=|x-y|$ be the distance map.  Our goal is to show that
\[
\hdim(\Delta(G)) \ge \chi - o_{T,\e,\tau}(1).
\]
Since $\Delta(G)\subset \Delta(A\times A)$, this will establish the Hausdorff dimension bound \eqref{eq:full-dist-Hausdorff-dim-claim}. In turn, by the standard pigeonholing argument, in order to complete the proof it is enough to prove the following claim.

\textbf{Claim}. The following holds if $\ell$ is large enough in terms of $\mu,\nu,s,T,\e,\tau$: if $H$ is a Borel subset of $[0,1)^2\times [0,1)^2$ such that $\mu_G(H) > \ell^{-2}$, then
\[
\log\cN(\Delta(H),\ell) \ge T\ell (\chi-o_{T,\e,\tau}(1)).
\]
We will in fact prove the following stronger fact: there is $z$ such that
\begin{equation} \label{eq:claim-full-dist-thm}
\log\cN(\{ |x-z|:(x,z)\in H \text{ or } (z,x)\in H\},\ell) \ge T\ell (\chi-o_{T,\e,\tau}(1)).
\end{equation}
We may assume also that  $H$ is a $2^{-T\ell}$-set, since proving \eqref{eq:claim-full-dist-thm}  for $\cR_\ell H$ also proves it for $H$. We also assume that $H\subset\cR_\ell G$. Note that $\cR_\ell\mu(H) = \mu(H) = \mu_G(H)\mu(G)> \ell^{-2}/3$.

We denote the decompositions given by Proposition \ref{prop:bourgain} applied to $\cR_\ell\mu_j$ by $(X_{j,\ell,k})$, and set $\rho_{j,\ell,k}=(\cR_\ell\mu_j)_{X_{j,\ell,k}}$. Write $\eta_{k_1,k_2}=\cR_\ell\mu_1(X_{1,\ell,k_1})\cR_\ell\mu_2(X_{2,\ell,k_1})$, and note that
\begin{align}
(\cR_\ell \mu_1\times\cR_\ell\mu_2)(H) &\ge  \sum_{k_1,k_2} \eta_{k_1,k_2} (\rho_{1,\ell,k_1}\times\rho_{2,\ell,k_2})(H) - 2\cdot 2^{-\e T\ell},
\end{align}
so we can fix $k_1, k_2$ such that $\rho_j=\rho_{j,\ell,k_j}$ satisfy
\begin{equation} \label{eq:H-large-mass}
( \rho_1\times\rho_2)(H)  \ge \ell^{-2}/4.
\end{equation}
We can find $y_j\in \supp(\rho_j)$ and sets $A_j\subset\supp(\rho_j)$ such that  $\rho_j(A_j)\ge \ell^{-2}/4$ and
\[
(x_1,y_2), (x_2,y_1) \in H\subset \cR_\ell(G) \quad \text{for all } x_j\in A_j
\]
Let $\sigma^{(j)}\in [-1,1]^\ell$ be the sequences such that $\rho_j$ is $\sigma^{(j)}$-regular. Using Lemma \ref{lem:energy-regular} and the fact that the sets $X_{j,\ell,k_j}$ are obtained from Proposition \ref{prop:bourgain}, we have
\[
  \sum_{i=1}^k \sigma_i^{(j)} \ge (s-1)k - \ell o_{T,\e}(1) \qquad(k=1,\ldots,\ell).
\]
We apply Lemma \ref{l:sigmatof} to obtain $1$-Lipschitz functions $f_1, f_2:[0,1]\to \R$ with
\begin{enumerate}[(a)]
  \item $f_j(k/\ell) = \tfrac{1}{\ell}(\sigma_1^{(j)}+\ldots+\sigma_k^{(j)})$ for $k\in [o_{T,\e}(1)\ell ,\ell]$,
  \item $f_j(x) \ge (s-1-o_{T,\e}(1))x$,
  \item  $\frac{1}{\ell} \M_\tau(\sigma^{(j)}) \le \T(f_j)+o_{T,\e,\tau}(1)$.
\end{enumerate}
We assume first that $\T(f_1)\le 1-\chi+2o_{T,\e}(1)/3$, where this instance of $o_{T,\e}(1)$ is the one from (b). In this case there is $y'$ with $|y'-y_2|\lesssim 2^{-T\ell}$ and such that $y'$ and $A_1$ satisfy the hypotheses of Proposition \ref{prop:box-counting} applied with $\mu=\mu_1$ and $\nu=\mu_2$ (recall the definition of the set $G$). From (c) and the proposition we get
\[
\log\cN(\Delta(G),\ell) \ge \log\cN(\Delta_{y'}(A_2),\ell) -O(1) \ge (\chi - o_{T,\e,\tau}(1))T\ell.
\]
The case $\T(f_1)\le 1-\chi+2o_{T,\e}(1)/3$ is identical, so we assume from now on that for both $j=1,2$ it holds that
\[
\T(f_j) > 1-\chi+2o_{T,\e}(1)/3 = \frac{1-2(u-o_{T,\e}(1))}{3} -\eta.
\]

The goal is to apply Corollary \ref{cor:GIOW} with $\mu=\rho'_1$, $\nu=\rho'_2$ and $s_1=5/4$, where $\rho'_j = \cR_{\ell'}\rho_j$ and $\ell'=\lfloor \xi \ell\rfloor$. To begin, notice that since the $\rho_j$ were obtained from Proposition \ref{prop:bourgain}, we have
\[
\cE_s(\rho'_j) \lesssim \cE_s(\rho_j) \lesssim 2^{o_{T,\e}(1)T\ell}.
\]
Now, thanks to (b) both $f_1$ and $f_2$ satisfy the hypotheses of Corollary \ref{c:2/57} with $u-o_{T,\e}(1)$ in place of $u$, so we get
\[
f_j(x) \ge \frac{x-3\eta(1-4u)}{4} - o_{T,\e}(1) \quad\text{on } [0,\xi]
\]
for $j=1,2$. Using (a) above and recalling that $\ell'=\lfloor \xi \ell\rfloor$, this yields
\[
\sum_{i=1}^k 1/4-\sigma_i^{(j)} \le \ell\left(\frac{3}{4}\eta(1-4u)+o_{T,\e}(1)\right),\quad k=1,\ldots, \ell'.
\]
Lemma \ref{lem:regular-mass-decay} applied to $\rho'_2$ yields
\begin{equation}  \label{eq:rho'-small-Frostman-exp}
\rho'_2(B(x,r)) \lesssim 2^{(3\eta(1-4u)/4+o_{T,\e}(1))T\ell} \, r^{5/4},
\end{equation}
while Lemma \ref{lem:energy-regular} shows that
\begin{equation} \label{eq:rho'-small-energy}
 \cE_{5/4+o_{T,\e}(1)}(\rho'_1) \lesssim 2^{o_{T,\e}(1)T\ell} \cE_{5/4}(\rho'_1) \lesssim 2^{(3\eta(1-4u)/4+o_{T,\e}(1))T\ell}.
\end{equation}
We also have the separation property
\[
\dist(\supp(\rho'_1),\supp(\rho'_2)) \gtrsim 1.
\]
We have verified that the assumptions of Corollary \ref{cor:GIOW} are satisfied for $\mu=\rho'_1$, $\nu=\rho'_2$, with $s_0=s$, $s_1=5/4$ and
\[
K_{\rho'_j,s} \lesssim 2^{o_{T,\e}(1) T\ell},\, \wt{K}_{\rho'_2,5/4} \lesssim 2^{T\ell(3\eta(1-4u)/4+o_{T,\e}(1))} .
\]
Let $C_1(\de),C_2$ be the numbers given by the corollary. Note that
\[
\rho'_1(\cR_{\ell'}A_1) = \cR_{\ell'}\rho_1(\cR_{\ell'}A_1) \ge \rho_1(A_1) \ge \ell^{-2}/5.
\]
On the other hand, we can choose $\delta = o_{T,\e}(1)$ and $R=2^{o_{T,\e}(1)T\ell}$ so that
\[
\ell^{-2}/5 > C_1(\de) (K_{\rho'_1,s} K_{\rho'_2,s} )^{1/2} R^{-\delta/C_2}\quad\text{and}\quad R^{-\delta/C_2}\le 2^{-T\ell/O_{T,\e}(1)}.
\]
It follows from \eqref{eq:H-large-mass} that we can find $B\subset \supp(\rho'_2)$ with
\[
\rho'_2(B)\ge \ell^{-2}/8 \gg R^{-\de/C_2}
\]
such that for all $y\in B$ it holds that
\[
\rho'_1(A_y)\ge \ell^{-2}/8 \quad\text{where } A_y=\{ x: (x,y)\in\cR_{\ell'}H\}.
\]
By Corollary \ref{cor:GIOW} and \eqref{eq:rho'-small-energy}, we can find a point $y\in B$ such that $\rho'_1(A_y) \gtrsim \ell^{-2}$ and
\begin{align*}
|\Delta_y(\cR_{\ell'}A_y)| &\ge 2^{-(\eta(1-4u)/4-o_{T,\e}(1))T\ell} \left( \cE_{5/4+o_{T,\e}(1)}(\rho'_1)\right)^{-1}\\
  & \gtrsim 2^{-(\eta(1-4u)+o_{T,\e}(1)) T\ell}.
\end{align*}
Using that $\dist(y,\supp(\nu))\lesssim 2^{-T\ell'}$, pick $z\in\supp(\nu)$ such that
\begin{align*}
\log\cN(\Delta_{z}(A_1),\ell) &\ge \log\cN(\Delta_{z}(A_1),\ell') \\
 &\ge (1-o_{T,\e}(1)) T\ell' - \eta(1-4u) T\ell \\
 &\ge (\xi-\eta(1-4u)-o_{T,\e}(1))T\ell.
\end{align*}
In light of \eqref{eq:identity-chi}, this yields \eqref{eq:claim-full-dist-thm} and concludes the proof of the Hausdorff dimension bound \eqref{eq:full-dist-Hausdorff-dim-claim}.

\subsection{Proof of Theorem \ref{thm:full} (box dimension part)}

Now we turn to the part of Theorem \ref{thm:full} concerning the lower box dimension of $\Delta_y(A)$. We now assume that $\mathcal{H}^s(A)>0$ and aim to prove that there is $y\in A$ such that
\begin{equation} \label{eq:lower-bound-box-dim}
\lbdim(\Delta_y(A)) \ge \chi(s-1).
\end{equation}
In fact, we will prove that $\nu$-almost all $y$ are good, where $\nu$ is an  $s$-Frostman measure on $A$, i.e. $\nu(B(x,r)) \lesssim r^s$.

As usual we assume that $A\subset [0,1)^2$, and allow all implicit constants to depend on $s$ and the implicity constant in the Frostman condition. Moreover, we fix $u=s-1$, $\chi=\chi(u)$ and $\eta,\xi$ as given by Corollary \ref{c:2/57}.

We claim that it is enough to show that for large enough $\ell$ (depending on $T,\e,\tau$) if $\nu(B)>\ell^{-2}$, then there is $y\in B$ such that
\begin{equation} \label{eq:claim-lower-bound-box-dim}
\log\cN(\Delta_y A,\ell) \ge (\chi-o_{T,\e,\tau}(1))T\ell.
\end{equation}
Indeed, assuming this, it follows that for $\nu$-almost all $y$ the inequality \eqref{eq:claim-lower-bound-box-dim} holds for all sufficiently large $\ell$ (depending on $T,\e,\tau$ and $y$). Taking a suitable sequence $(T_j,\e_j,\tau_j)$ we see that \eqref{eq:lower-bound-box-dim} holds for $\nu$-almost all $y$.

Before we embark on the proof of \eqref{eq:claim-lower-bound-box-dim}, we present a lemma that will allow us to go back to a setting very similar to that of \S\ref{subsec:full-Hausdorff}. By a square we mean one with axes-parallel sides, and we denote side-lengths by $\lambda(\cdot)$.
\begin{lemma} \label{lem:separated-squares}
Suppose $\rho\in\cP([0,1)^2)$ satisfies $\rho(B(x,r)) \le L\, r^s$ for some $s\in (1,2]$, $L>0$. Then there exists squares $Q, Q_1, Q_2$ such that $Q_j\subset Q$, $\rho(Q_j)\ge L^{-O_s(1)}$ and
\[
\dist(Q_1,Q_2) \gtrsim \lambda(Q).
\]
We emphasize that the implicit constants do not depend on $L$.
\end{lemma}
\begin{proof}
Let $K=K(s)\gg 1$ be a number to be determined later. Let us say that a square $Q$ is \emph{good} if it contains two sub-squares $Q_1,Q_2$ with $\lambda(Q_j)\ge \lambda(Q)/100$, $\rho(Q_j)\ge \rho(Q)/K$ and $\dist(Q_1,Q_2)\ge \lambda(Q)/100$; otherwise, we say that $Q$ is bad.

We claim that if $Q$ is bad then it contains a sub-square $Q'$ with $\lambda(Q')\le \lambda(Q)/10$ and $\rho(Q')\ge \rho(Q)(1-O(1/K))$. Indeed, if $K$ is large enough then certainly there is a sub-square $Q_0$ with $\lambda(Q_0)=\lambda(Q)/100$ and $\rho(Q_0)\ge \rho(Q)/K$. Let $Q'$ be the smallest square that contains all sub-squares of $Q$ with side-length $\lambda(Q)/100$ that are at distance $\le 1/100$ of $Q_0$. Then $\lambda(Q')\le\lambda(Q)/10$ and $Q\setminus Q'$ can be covered by $O(1)$ squares of side length $\lambda(Q)/100$, each of which has $\rho$-mass at most $\rho(Q)/K$, so the claim follows.

We now construct inductively a decreasing sequence of squares $(R_j)$ as follows. Let $R_0=[0,1)^2$.  If $R_j$ is good we stop. Otherwise, we  let $R_{j+1}$ be a subsquare of $R_j$ with $\lambda(R_{j+1})\le \lambda(R_j)/10$ and $\rho(R_{j+1})\ge \rho(R_j)(1-O(1/K))$.  By construction, if $K$ was taken large enough in terms of $s$ then
\begin{align*}
\rho(R_j) &\ge (1-O(1/K))^j \\
&\ge \lambda(R_j)^{\log(1-O(1/K))/\log(1/10)} \\
&\ge 100 \lambda(R_j)^{-1} \lambda(R_j)^s,
\end{align*}
and this contradicts the Frostman assumption if $\lambda(R_j)\le L^{-1}$. Hence the process must stop in $\lesssim \log L$ steps. Letting $Q=R_j$, the last square obtained in this process, we obtain the claim.
\end{proof}

Now we begin the proof of \eqref{eq:claim-lower-bound-box-dim}. Let $B$ and $\ell$ with $\nu(B)>\ell^{-2}$ be given. We apply Lemma \ref{lem:separated-squares} to $\nu_B$, with $L=O(\ell^{2})$; let $Q,Q_1,Q_2$ be the squares given by the lemma. Moreover, let $H$ be the homothety mapping $Q$ to $[0,1)^2$, which has ratio $\le \ell^{O(1)}$ by the lemma and the Frostman condition. Finally, let $\mu_j = H(\nu_{Q_j\cap B})$. By construction, we have
\[
\mu_j(B(x,r)) \lesssim \ell^{O(1)} \, r^s,
\]
and
\[
\dist(\supp(\mu_1),\supp(\mu_2)) \gtrsim 1.
\]
These properties of $\mu_1,\mu_2$ are very similar to those in \S\ref{subsec:full-Hausdorff}, except that we have the extra factor $\ell^{O(1)}$ in the Frostman constant (and hence also in the $s_0$-energy of $\mu_j$, where $s_0=(1+s)/2$). However, since $\ell^{O(1)}\ll 2^{\e T\ell}$ for large $\ell$, the proof of the Hausdorff dimension part carries over in exactly the same way to yield (for $j$ equal to either $1$ or $2$) a point $z\in \supp\mu_j$ such that
\[
\log\cN(\Delta_z(\supp(\mu_{2-j})),\ell) \ge (\chi-o_{T,\e,\tau}(1))T\ell.
\]
Applying $H^{-1}$ to $\Delta_z(\supp(\mu_{2-j}))$, we conclude that $y=H^{-1}z$ satisfies
\begin{align*}
\log\cN(\Delta_y(B),\ell) &\ge \log\cN(\Delta_y(B),\ell+O(\log\ell))-O(\log\ell) \\
&\ge \log\cN(\Delta_z(\supp(\mu_{2-j})),\ell)-O(\log\ell) \\
&\ge (\chi-o_{T,\e,\tau}(1))T\ell.
\end{align*}
Since $B\subset A$, this concludes the proof of \eqref{eq:claim-lower-bound-box-dim} and with it of Theorem \ref{thm:full}.


\end{document}